\documentclass[a4paper,10pt,article]{amsart}
\usepackage{amsmath, amscd,amssymb,amsthm,CJKnumb,verbatim,indentfirst,dsfont,mathrsfs,ipa,graphicx,textcomp,enumerate}
\usepackage[all]{xy}
\usepackage{CJK,ipa,amsmath,booktabs,longtable}
\usepackage[numbers]{natbib}
\renewcommand{\parallel}{\hbox{/\kern -2pt/}}
\newtheorem{theorem}{Theorem}[section]
\newtheorem{lemma}[theorem]{Lemma}

\newtheorem{conjecture}[theorem]{Conjecture}

\theoremstyle{definition}
\newtheorem{definition}[theorem]{Definition}
\newtheorem{example}[theorem]{Example}

\theoremstyle{remark}
\newtheorem{remark}[theorem]{Remark}

\numberwithin{equation}{section}

\title{Uniformly bounded fibred coarse embeddability and uniformly bounded a-T-menability}
\begin{document}

\author{Jianguo Zhang}
\address{Research Center for Operator Algebras, East China Normal University}
\email{jgzhang@math.ecnu.edu.cn}

\author{Dapeng Zhou}
\address{\parbox{\linewidth}{School of Statistics and Information, Shanghai University of International Business and Economics.\\ Research Center for Operator Algebras, East China Normal University.}}

\email[Corresponding Author]{giantroczhou@126.com}

\maketitle

\begin{abstract}In this paper, we introduce the concept of uniformly bounded fibred coarse embeddability of metric spaces, generalizing the notion of fibred coarse embeddability defined by X. Chen, Q. Wang and G. Yu. Moreover, we show its relationship with uniformly bounded a-T-menability of groups. Finally, we give some examples to illustrate the differences between uniformly bounded fibred coarse embeddability and fibred coarse embeddability.
\end{abstract}
\pagestyle{plain}
\section{Introduction}
    For a second countable, locally compact group $G$, recall that $G$ is said to be \textit{a-T-menable} if it admits a continuous, isometric affine action $\alpha$ on a separable Hilbert space $H$ which is metrically proper, that is, for any bounded subset $B$ of $H$, the set $\{g\in G: \alpha_g(B)\cap B \neq \emptyset\}$ is relatively compact in $G$. The a-T-menability (also known as Haagerup property) introduced by M. Gromov is the strong opposition of Kazhdan's property (T) and includes amenable groups \cite{Lin}, finitely generated free groups, $SO(n,1)$ and $SU(n,1)$ (the isometry groups of $n$-dimensional real and complex hyperbolic spaces respectively), see \cite{HaagerupProperty}\cite{Gromovasymptotic}. The most spectacular result about the a-T-menability was obtained by N. Higson and G. Kasparov \cite{HigsonKasparov} and J.-L. Tu \cite{TuATmenable}: the Baum-Connes conjecture holds for all a-T-menable groups. \par
    For a metric space $(X,d)$, recall that $X$ is said to admit a \textit{coarse embedding into Hilbert space} if there exists a map $f:X\rightarrow H$ and two non-decreasing functions $\rho_{-}$ and $\rho_{+}$ from $[0,\infty)$ to $(-\infty,\infty)$ with $\lim_{t\rightarrow \infty}\rho_{\pm}(t)=\infty $ such that 
    $$\rho_{-}(d(x,y))\leq ||f(x)-f(y)||\leq \rho_{+}(d(x,y))$$
for any $x,y\in X$. The coarse embeddability was first defined by M. Gromov in \cite{Gromovasymptotic}, and it contains a large class of metric spaces, e.g. $n$-dimensional Euclidean spaces and hyperbolic spaces. A finitely generated group $\Gamma$ with a symmetric generating set $S$ can be seen as a metric space with word length metric defined by $d(\gamma_1,\gamma_2)=\min\{n:\gamma_1^{-1}\gamma_2=\sigma_1\cdots \sigma_n,\sigma_i\in S\}$ for $\gamma_1,\gamma_2\in \Gamma$. In that sense, any finitely generated a-T-menable group can be coarsely embedded into Hilbert space (see Lemma \ref{ActionImplyEmbed}). A remarkable result about coarse embeddability is that coarse Baum-Connes conjecture holds for any space with bounded geometry,  which admits a coarse embedding into Hilbert space, proved by G. Yu \cite{YuEmbedding} and G. Skandalis, J.-L. Tu and G. Yu \cite{SkandalisTuYu}. \par
    Is there a bridge from groups to metric spaces connecting the above two notions? It was first observed by J. Roe in \cite{Roe03} that a finitely generated, residually finite group $\Gamma$ is a-T-menable if its box space $\square\Gamma$ (see Definition \ref{DefBoxSpace}) admits a coarse embedding into Hilbert space. However, the converse of this statement is false, e.g. free group $\mathbb{F}_2$, which admits a box space consisting of a sequence of expanders, and any sequence of expanders cannot coarsely embed into Hilbert space \cite{NowakYu}\cite{WillettYuBook}. In \cite{ChenWangYufibredcoarse}, X. Chen, Q. Wang, and G. Yu introduced the concept of \textit{fibred coarse embeddability}, generalizing the notion of coarse embeddability, to detect the maximal coarse Baum-Connes conjecture. Shortly afterwards, X. Chen, Q. Wang and X. Wang in \cite{ChenWangWangFibredCoarse} established that $\Gamma$ is a-T-menable if and only if $\square\Gamma$ admits a fibred coarse embedding into Hilbert space. This result was also proved by M. Finn-Sell in \cite{Finn-SellFibred} using groupoid method. The above statements had been generalized to the case of $L^p$-spaces by S. Arnt \cite{ArntFibredLp}, to the case of uniformly convex Banach spaces by G. Li and X. Wang \cite{LiWangFibred}. If we consider warped cones (see Definition \ref{DefWarpedCones}) rather than box spaces, there are some similar results obtained by J. Roe \cite{RoeWarpedCone}, D. Sawicki and J. Wu \cite{SawickiWu}, Q. Wang and Z. Wang \cite{2WangWarpedCone}. \par
    A representation $\pi$ of a locally compact group $G$ on a Hilbert space $H$ is \textit{uniformly bounded}, if there exists a constant $L\geq 1$ such that $||\pi(g)||\leq L$ for any element $g$ in $G$. In \cite{Julg1994}\cite{Julg2019}, P. Julg suggested that Baum-Connes conjecture should be solved for real rank one Lie groups by using M. Cowling's strip of uniformly bounded representations \cite{CowlingU.B.Rep}. If we replace the ``isometric" action with the ``uniformly bounded" action in the definition of a-T-menability, we can obtain a more general concept called \textit{uniformly bounded a-T-menability} (see Definition \ref{UBa-T-menable}). And Y. Shalom conjectured that any hyperbolic group has uniformly bounded a-T-menability \cite{NowakGroupAction}. Recently, S. Nishikawa verified this conjecture for $Sp(n,1)$ (the isometry groups of $n$-dimensional quaternion hyperbolic spaces) \cite{NishikawaSpn1}. It was kind of a surprise, since $Sp(n,1)$ has property (T), whence cannot be a-T-menable when $n$ is larger than 2. \par
    In this paper, we introduce the concept of \textit{uniformly bounded fibred coarse embeddability} of metric spaces (see Definition \ref{ubfibred}) generalizing the concept of fibred coarse embeddability and discuss its relationship with the concept of uniformly bounded a-T-menability of groups. That is, for box spaces (see Definition \ref{DefBoxSpace}), we obtain the following theorem:
\begin{theorem}(see Theorem \ref{MainTheorem})
Let $\Gamma$ be a finitely generated residually finite group with a sequence of finite index normal subgroups $\Gamma=\Gamma_0\unrhd \Gamma_1\unrhd \cdots \unrhd \Gamma_n \unrhd \cdots$ such that $\bigcap_{i\in\mathbb{N}}\Gamma_i=\{e\}$. Let $\square\Gamma$ be the box space associated with this sequence. Then $\Gamma$ is uniformly bounded a-T-menable if and only if $\square\Gamma$ admits a uniformly bounded fibred coarse embedding into Hilbert space.
\end{theorem}
And for warped cones (see Definition \ref{DefWarpedCones}), we have a similar theorem:
\begin{theorem} (see Theorem \ref{MainTheorem2})
Let $\Gamma \curvearrowright Y$ be an action of a finitely generated group on a compact metric space.
\begin{enumerate}
\item Assume the action is free and linearizable in a Hilbert space if $\Gamma$ is uniformly bounded a-T-menable, then $\mathcal{O}_{\Gamma}Y$ admits a uniformly bounded fibred coarse embedding into Hilbert space.
\item Assume the action admits an invariant Borel probability measure $\mu$ on $Y$ and is essentially free with respect to $\mu$. If $\mathcal{O}_{\Gamma}Y$ admits a uniformly bounded fibred coarse embedding into Hilbert space, then $\Gamma$ is uniformly bounded a-T-menable.
\end{enumerate}
\end{theorem}
We should note that the above two theorems also hold for uniformly convex Banach spaces, e.g. $L^p$-spaces. Finally, we will give some examples to illustrate the differences between uniformly bounded fibred coarse embeddability and fibred coarse embeddability. \par

    The paper is organized as follows: In Section 2, we recall some results about uniformly bounded a-T-menable groups. In Section 3, we introduce the concept and properties of uniformly bounded fibred coarse embeddability of metric spaces. In Section 4, we state and prove the above two main theorems. In the end, we give some examples. 
    
\begin{remark}
For simplicity of notation, we will write ``u.b." instead of ``uniformly bounded" in the rest of this paper.
\end{remark} 

\subsection*{Acknowledgements}
    The authors would like to thank Qin Wang, Guoliang Yu and Jiawen Zhang for providing many valuable suggestions and comments. This work is partially supported by NSFC (No. 11771143, No.11801178).

\section{uniformly bounded a-T-menability}
In this section, we recall the concept of \textit{uniformly bounded a-T-menability} for groups, generalizing the notion of \textit{a-T-menability} defined by M. Gromov \cite{Gromovasymptotic}\cite{HaagerupProperty}. Then we discuss the coarse embeddability of  u.b. a-T-menable groups.
\begin{definition}\label{def:affinemap}
A map $\psi:H \rightarrow H$ is called an \textit{affine map} if there exists a linear map $T$ on $H$ and an element $x\in H$ such that $\psi(h)=T(h)+x$ for any $h\in H$. The \textit{norm} of $\psi$, denoted $\|\psi\|$, is defined to be the norm of $T$. And $\psi$ is called an $\textit{affine isomorphism}$, if there exists an affine map $\phi$ on $H$ such that $\psi\phi=\phi\psi=I$, where $I$ is identity map on $H$. Such $\phi$ is unique if $\phi$ exists, and thus we may denote $\phi$ by $\psi^{-1}$. Let Affine($H$) be the set consisting of all affine isomorphisms on $H$.
\end{definition}
\begin{definition}\label{UBa-T-menable}
Let $G$ be a second countable, locally compact group, say $G$ is \textit{uniformly bounded a-T-menable}, if there exists a continuous affine action $\alpha: G\rightarrow$ Affine($H$) satisfying,
\begin{itemize}
\item it is a uniformly bounded action, i.e. there exists a constant $L>0$ such that $||\alpha_g||\leq L$ for all $g\in G$;
\item it is a metrically proper action, i.e. for all bounded subsets $B$ of $H$, the set $\{g\in G: \alpha_g(B)\cap B\not=\emptyset\}$ is relatively compact in $G$. 
\end{itemize}
In particular, say $G$ is \textit{a-T-menable} if there exists a continuous isometric action which satisfies the above two conditions.
\end{definition}

We have the following conjecture about u.b. a-T-menability (see Conjecture 35 of \cite{NowakGroupAction}).
\begin{conjecture} {(Y. Shalom)}
Every hyperbolic group is u.b. a-T-menable. 
\end{conjecture}

Recently, in \cite{NishikawaSpn1}, S. Nishikawa verified the above conjecture for Lie groups $Sp(n,1)$ (see Example \ref{Example2} for precise definition).
\begin{theorem}{(S. Nishikawa)} \label{NishikawaTheorem}
For any $n\geq 1$, the simple rank one Lie group $Sp(n,1)$ is u.b. a-T-menable.
\end{theorem}

\begin{remark}
We notice that the group $Sp(n,1)$ has property(T) for $n\geq 2$ (see \cite{PropertyT}), thus it cannot be an a-T-menable group. However, the above theorem tells us that there exists a constant $L>1$ such that it has a uniformly bounded (with upper bound $L$), metrically proper action on Hilbert space.  
\end{remark}

We finish this section by considering coarse embeddability of u.b. a-T-menable groups.
\begin{lemma}\label{ActionImplyEmbed}
Let $\Gamma$ be a finitely generated group equipped with a word length metric. If $\Gamma$ is u.b. a-T-menable, then $\Gamma$ admits an equivariant coarse embedding into Hilbert space. 
\end{lemma}
\begin{proof}
Let $\alpha$ be a uniformly bounded, metrically proper affine action of $\Gamma$ on $H$, then 
      $$\alpha_{\gamma}(h)=T_{\gamma}(h)+b_{\gamma}$$
for all $\gamma\in \Gamma$ and $h\in H$, where $T_{\gamma}$ is a linear operator on $H$ satisfying $\sup_{\gamma\in \Gamma}||T_{\gamma}||\leq L$ for some constant $L$ and $\lim_{|\gamma|\rightarrow \infty}||b_{\gamma}||=\infty$.\par
Define a map $f: \Gamma \rightarrow H$ by 
      $$f(\gamma)=b_{\gamma},$$
then we have $f(\gamma'\gamma)=\alpha_{\gamma'}f(\gamma)$ which means that $f$ is an equivariant map. For each $t\in [0,\infty)$, define
      $$\rho_{-}(t)=L^{-1}\inf\{||b_{\gamma}||: |\gamma|\geq t\}$$
      $$\rho_{+}(t)=L\sup\{||b_{\gamma}||: |\gamma|\leq t\}.$$
Then for any two elements $\gamma_1,\gamma_2\in \Gamma$, we have $$||f(\gamma_1)-f(\gamma_2)||=||b_{\gamma_2(\gamma^{-1}_2\gamma_1)}-b_{\gamma_2}||=||T_{\gamma_2}(b_{\gamma^{-1}_2\gamma_1})||.$$ 
Thus
      $$||f(\gamma_1)-f(\gamma_2)||\in [L^{-1}||b_{\gamma^{-1}_2\gamma_1}||,L||b_{\gamma^{-1}_2\gamma_1}||]\subset[\rho_{-}(d(\gamma_1,\gamma_2)),\rho_{+}(d(\gamma_1,\gamma_2))].$$
Moreover, $\rho_{\pm}$ are two non-decreasing functions and $\lim_{t\rightarrow \infty}\rho_{\pm}(t)=\infty$ since $\lim_{|\gamma|\rightarrow \infty}||b_{\gamma}||=\infty$. Therefore $\Gamma$ admits an equivariant coarse embedding into Hilbert space. 
\end{proof}

\section{uniformly bounded fibred coarse embeddability}

In this section, we introduce the concept of \textit{uniformly bounded fibred coarse embedding into Hilbert space}, generalizing the notion of \textit{fibred coarse embedding into Hilbert space} defined by X. Chen, Q. Wang and G. Yu \cite{ChenWangYufibredcoarse}. Let $H$ be a separable Hilbert space, as a model space.
\begin{definition}\label{ubfibred}
Let $(X,d)$ be a metric space. $X$ is said to admit a \textit{uniformly bounded fibred coarse embedding into Hilbert space} if there exists a constant $L>0$ and
\begin{itemize}
\item a field of Hilbert spaces $(H_x)_{x\in X}$ over $X$;
\item a section $s:X\rightarrow \bigsqcup_{x\in X}H_x$ (i.e. $s(x)\in H_x$);
\item two non-decreasing functions $\rho_{-}$ and $\rho_{+}$ from $[0,\infty)$ to $(-\infty,\infty)$ with $\lim_{t\rightarrow \infty}\rho_{\pm}(t)=\infty $
\end{itemize} 
such that for any $r>0$, there exists a bounded subset $K\subset X$ for which there exists a ``trivialization"
                    $$t_C:(H_x)_{x\in C}\rightarrow C \times H$$
for each subset $C\subset X\setminus K$ of diameter less than $r$ (i.e. a map from $(H_x)_{x\in C}$ to the constant field $C\times H$ over $C$ such that the restriction of $t_C$ to the fiber $H_x$ is an affine isomorphism $t_C(x): H_x \rightarrow H$ with $\max\{||t_C(x)||,||t_C(x)^{-1}||\}\leq L$ for any $x\in C$), satisfying 
\begin{enumerate}
\item $\rho_{-}(d(x,y))\leq ||t_C(x)(s(x))-t_C(y)(s(y))||\leq \rho_{+}(d(x,y))$ for any $x,y \in C$;
\item for any two subsets $C_1,C_2\subset X\setminus K$ of diameter less than $r$ with $C_1 \cap C_2 \neq \emptyset$, there exists an affine isomorphism $t_{C_1C_2}$ on $H$ with $\max\{||t_{C_1C_2}||,||t_{C_1C_2}^{-1}||\}$ less than $L$ such that $t_{C_1}(x)t_{C_2}(x)^{-1}=t_{C_1C_2}$ for all $x\in C_1 \cap C_2$.
\end{enumerate}
In particular, $X$ is said to admit a \textit{fibred coarse embedding into Hilbert space} when $L=1$ as above.
\end{definition}   

Let $\Gamma$ be a finitely generated group with a symmetric generating set $S$. Define a \textit{word length function} $|\cdot|$ on $\Gamma$ by $|\gamma|=\min\{n:\gamma=\sigma_1 \cdots \sigma_n, \sigma_i\in S\}$ for $\gamma \in \Gamma$. It induces a \textit{word length metric} on $\Gamma$ by $d(\gamma_1,\gamma_2)=|\gamma_1^{-1}\gamma_2|$ for any $\gamma_1,\gamma_2\in \Gamma$. Note that two word length metrics defined by two different symmetric generating sets are bi-Lipschitz equivalent (see Chapter 1 of \cite{NowakYu}). 

    Obviously, coarse embeddability implies u.b. fibred coarse embeddability. However, the following theorem tells us that for normed space or a finitely generated group equipped with a word length metric, these two concepts are identical. Nevertheless, there are many metric spaces which admit u.b. fibred coarse embedding into Hilbert space but do not admit coarse embedding into Hilbert space. We will give such examples in Section \ref{Examples}.
    
\begin{theorem}[\cite{DGLY-uniformembed}]
Let $X$ be a metric space. If there exist two non-decreasing functions $\rho_{-}$ and $\rho_{+}$ from $[0,\infty)$ to $(-\infty,\infty)$ with $\lim_{t\rightarrow \infty}\rho_{\pm}(t)=\infty $ such that for every finite subset $A\subset X$, there exists a map $f_A: A\rightarrow H$ satisfying 
    $$\rho_{-}(d(x,y))\leq ||f_A(x)-f_A(y)||\leq \rho_{+}(d(x,y))$$
for any $x,y\in A$, then $X$ admits a coarse embedding into a Hilbert space with bounds $\rho_{\pm}$.
\end{theorem}

We now consider when a metric space admits a u.b. fibred coarse embedding into Hilbert space. Let us first recall some notions.

\begin{definition}[\cite{NowakYu}]
Let $(X_n)_{n\in \mathbb{N}}$ be a sequence of bounded metric spaces, a \textit{coarse disjoint union} of $(X_n)_{n\in \mathbb{N}}$ is the disjoint union $X=\bigsqcup_{n\in \mathbb{N}}X_n$ equipped with a metric $d$ such that:
\begin{itemize}
\item the restriction of $d$ to each $X_n$ is the original metric on $X_n$; 
\item $d(X_i,X_j)\geq i+j$ if $i\not= j$.
\end{itemize}
Note that any two metrics on $X$ satisfying the above two conditions are coarsely equivalent.
\end{definition}

\begin{definition}
A metric space $\Tilde{X}$ is called a \textit{Galois covering} of a metric space $X$ if there exists a discrete group $\Gamma$ acting on $\Tilde{X}$ freely and properly by isometries such that $X=\Tilde{X}/\Gamma$. Denote by $\pi: \Tilde{X}\rightarrow X$ the associated covering map.   
\end{definition}

\begin{definition}[\cite{WillettYuExpander}] \label{asym-faithful}
A sequence of Galois coverings $(\Tilde{X}_n)_{n\in \mathbb{N}}$ of $(X_n)_{n\in \mathbb{N}}$ is said to be \textit{asymptotically faithful} if for any $r>0$ there exists $N\in \mathbb{N}$ such that the covering map $\pi_{n}:\Tilde{X}_n \rightarrow X_n$ is ``$r$-isometric" (i.e. for any subset $\Tilde{C}\subset \Tilde{X}_n$ of diameter less than $r$, the restriction of $\pi_n$ to $\Tilde{C}$ is an isometric map) for all $n\geq N$.  
\end{definition}

\begin{definition}\label{EquivariantCoarse}
Let $L$ be a positive constant, $(\Tilde{X}_n)_{n\in \mathbb{N}}$ be a sequence of Galois coverings of $(X_n)_{n\in \mathbb{N}}$ with $X_n=\Tilde{X}_n/\Gamma_n$. Say $(\Tilde{X}_n)_{n\in \mathbb{N}}$ admits an \textit{$L$-controlled uniform equivariant coarse embedding into Hilbert space}, if there exists a map $f_n:\tilde{X}_n\rightarrow H$ for each $n\in \mathbb{N}$ and two non-decreasing functions $\rho_{-}$ and $\rho_{+}$ from $[0,\infty)$ to $(-\infty,\infty)$ with $\lim_{t\rightarrow \infty}\rho_{\pm}(t)=\infty $ such that
\begin{itemize}
\item for each $n\in \mathbb{N}$, there exists an affine action $\alpha^{(n)}: \Gamma_n \rightarrow$ Affine($H$) satisfying $||\alpha^{(n)}_{\gamma}||\leq L$ and $f_n(\gamma\Tilde{x})=\alpha^{(n)}_{\gamma}(f_n(\Tilde{x}))$ for any $\gamma \in \Gamma_n$, $\Tilde{x}\in \Tilde{X}_n$;  
\item $\rho_{-}(d(\Tilde{x},\Tilde{y}))\leq ||f_n(\Tilde{x})-f_n(\Tilde{y})|| \leq \rho_{+}(d(\Tilde{x},\Tilde{y}))$ for any $\Tilde{x},\Tilde{y}\in \Tilde{X}_n$, $n\in \mathbb{N}$.		
\end{itemize}
\end{definition}

\begin{theorem}\label{EquivariantAndFibred}
Let $X=\bigsqcup_{n\in \mathbb{N}}X_n$ be a coarse disjoint union of a sequence of bounded metric spaces, if there exists a sequence of asymptotically faithful Galois coverings $(\Tilde{X}_n)_{n\in \mathbb{N}}$ which admits an $L$-controlled uniform equivariant coarse embedding into Hilbert space, then $X$ admits a u.b. fibred coarse embedding into Hilbert space.
\end{theorem}
\begin{proof}
For each $n\in \mathbb{N}$, by Definition \ref{EquivariantCoarse}, there exists an affine action $\alpha^{(n)}$ of $\Gamma_n$ on $H$ with $||\alpha^{(n)}_{\gamma}||\leq L$ for any $\gamma \in \Gamma_n$. For each $h\in H$, define 
    $$||h||_{\alpha^{(n)}}=\sup_{\gamma\in \Gamma_n}||\alpha^{(n)}_{\gamma}(h)-\alpha^{(n)}_{\gamma}(0)||.$$ 
Then $L^{-1}||h||\leq ||h||_{\alpha^{(n)}}\leq L||h||$, hence $||\cdot||_{\alpha^{(n)}}$ is a norm on $H$ which is equivalent to the original norm. Denote the Banach space $H$ equipped with the norm $||\cdot||_{\alpha^{(n)}}$ by $H_{\alpha^{(n)}}$, then the action $\alpha^{(n)}$ of $\Gamma_n$ on $H$ naturally induces an isometric affine action of $\Gamma_n$ on $H_{\alpha^{(n)}}$, also denoted by $\alpha^{(n)}$. \par
For each $n\in \mathbb{N}$, the group $\Gamma_n$ acts on $\Tilde{X}_n \times H_{\alpha^{(n)}}$ by $\gamma \cdot(\Tilde{x},h)=(\gamma \Tilde{x},\alpha^{(n)}_{\gamma}(h))$ for any $\Tilde{x}\in \Tilde{X}_n, h\in H_{\alpha^{(n)}}$ and $\gamma\in \Gamma_n$. For each $x\in X_n$, define
      $$H_x=(\pi^{-1}_n(x)\times H_{\alpha^{(n)}})/\Gamma_n$$
where $\pi^{-1}_n(x)$ is the preimage of $x$ by the covering map $\pi_n: \Tilde{X}_n\rightarrow X_n$. Fixing a point $z_x\in \pi^{-1}_n(x)$, define a norm on $H_x$ by 
      $$||[(\Tilde{x},h)]||_{z_x}=||\alpha^{(n)}_{\gamma}(h)||_{\alpha^{(n)}}$$
where $[(\Tilde{x},h)]$ is the orbit of $(\Tilde{x},h)\in (\pi^{-1}_n(x)\times H_{\alpha^{(n)}})$ and $\gamma\in \Gamma_n$ with $z_x=\gamma\Tilde{x}$. This norm induces the following metric on $H_x$
      $$d_x([(\Tilde{x},h)],[(\Tilde{x}',h')])=||\alpha^{(n)}_{\gamma_1}(h)-\alpha^{(n)}_{\gamma_2}(h')||_{\alpha^{(n)}}$$
where $\gamma_1,\gamma_2\in \Gamma_n$ with $z_x=\gamma_1\Tilde{x}$ and $z_x=\gamma_2\Tilde{x}'$. Note that $d_x$ is independent of $z_x$ since the action $\alpha^{(n)}$ is an isometric affine action on $H_{\alpha^{(n)}}$. Thus
      $$(H_x)_{x\in X_n}=(\Tilde{X}_n\times H_{\alpha^{(n)}})/\Gamma_n$$
is a field of Banach spaces over $X_n$ on which the norm of each fiber is equivalent to a Hilbert space norm. \par
For each $n\in \mathbb{N}$, define a section 
      $$s:X_n\rightarrow (\Tilde{X}_n\times H_{\alpha^{(n)}})/\Gamma_n$$ 
by
      $$s(x)=[(\Tilde{x},f_n(\Tilde{x}))]$$
for all $x\in X_n$, where $\Tilde{x}\in \pi^{-1}_n(x)$ and $f_n$ is an $L$-controlled uniform equivariant coarse embedded map in Definition \ref{EquivariantCoarse}. This is well defined since $f_n$ is an equivariant map. \par
For any $r>0$,  $X=\bigsqcup_{n\in \mathbb{N}}{X_n}$ is a coarse disjoint union and $(\Tilde{X}_n)_{n\in \mathbb{N}}$ is a sequence of asymptotically faithful Galois coverings, and thus there exists a positive integer $N$ such that 
\begin{enumerate}
\item $d(X_i,X_j)>r$ for any $i,j>N$;
\item the covering map $\pi_{n}:\Tilde{X}_n \rightarrow X_n$ is ``$r$-isometric" for any $n>N$.
\end{enumerate}
Let $K=\bigsqcup_{n\leq N}X_n$, which is a bounded subset in $X$ since every $X_n$ is a bounded metric space. Then for any $C\subset X\setminus K$ of diameter less than $r$, $C$ belongs to $X_n$ for only one integer $n>N$. For a point $z\in \pi^{-1}_n(C)$, let $\Tilde{C}_z$ denote the component of $\pi^{-1}_n(C)$ containg $z$, then each such component $\Tilde{C}_z$ gives rise to a trivialization:  
      $$t_{C,z}: (H_x)_{x\in C}=(\Tilde{C}\times H_{\alpha^{(n)}})/ \Gamma_n \rightarrow C\times H_{\alpha^{(n)}}$$
i.e. for each $x\in C$, define 
      $$t_{C,z}(x):H_x= (\pi^{-1}_n(x)\times H_{\alpha^{(n)}})/\Gamma_n\rightarrow H_{\alpha^{(n)}}$$ 
by
      $$t_{C,z}(x)([(\Tilde{x},h)])=\alpha^{(n)}_\gamma(h) $$
where $[(\Tilde{x},h)]$ is the orbit of $(\Tilde{x},h)\in (\pi^{-1}_n(x)\times H_{\alpha^{(n)}})$ and $\gamma$ is the unique element in $\Gamma_n$ satisfying $\gamma \Tilde{x}\in \Tilde{C}_z$. Obviously, $t_{C,z}(x)$ is an isometric affine isomorphism on $H_{\alpha^{(n)}}$, this implies that the map $t_{C,z}(x):(\pi^{-1}_n(x)\times H)/\Gamma_n\rightarrow H$ is an affine isomorphism with $\max\{||t_{C,z}(x)||,||t_{C,z}(x)^{-1}||\}\leq L$ for any $x\in C$.  \par
Now for any two elements $x,y\in C$, there exist two elements $\gamma_1,\gamma_2\in \Gamma_n$ such that $\gamma_1\Tilde{x}\in \Tilde{C}_z$ and $\gamma_2\Tilde{y}\in \Tilde{C}_z$, so that 
      $$||t_{C,z}(x)(s(x))-t_{C,z}(y)(s(y))||=||f_n(\gamma_1\Tilde{x})-f_n(\gamma_2\Tilde{y})||$$ 
belongs to $[\rho_-(d(x,y)),\rho_+(d(x,y))]$, where $||\cdot||$ is the original norm on the Hilbert space $H$ and the two functions $\rho_{\pm}$ come from Definition \ref{EquivariantCoarse}.\par
Moreover, for any two subsets $C_1,C_2\subset X\setminus K$ of diameter less than $r$ with $C_1 \cap C_2 \neq \emptyset$, there exists a unique integer $n>N$ such that $C_1,C_2\in X_n$. Then for any $x\in C_1\cap C_2$, we have 
      $$t_{C_1,z_1}(x)t_{C_2,z_2}(x)^{-1}=\alpha^{(n)}_\gamma$$
where $\gamma$ is the unique element in $\Gamma_n$ satisfying $\gamma \Tilde{C}_{1,z_1}\cap \Tilde{C}_{2,z_2}\not= \emptyset$. It is an affine isometry on $H_{\alpha^{(n)}}$, and hence is an affine isomorphism on $H$ satisfying $\max\{||\alpha^{(n)}_\gamma||,||(\alpha^{(n)}_\gamma)^{-1}||\}\leq L$. This completes the proof.
\end{proof}

\begin{remark}
In particular, if $L=1$ in the assumption of the above theorem, then we can obtain that $X$ admits a fibred coarse embedding into Hilbert space. 
\end{remark}

\begin{remark}
The above proof also holds for general Banach spaces. Thus the above theorem is true for all Banach spaces. 
\end{remark}

\section{main results}
In this section, we will characterize u.b. a-T-menability of groups by u.b. fibred coarse embeddability of box spaces (Definition \ref{DefBoxSpace}) and warped cones (Definition \ref{DefWarpedCones}), respectively, i.e. Theorem \ref{MainTheorem} and Theorem \ref{MainTheorem2}, the proofs of which are based on the ultraproduct machinery developed by S. Arnt \cite{ArntFibredLp}. Before doing that, we need to give some preliminaries.

We first recall the concept of \textit{ultraproducts} for a family of Hilbert spaces. 
\begin{definition}
A \textit{non-principal ultrafilter} on $\mathbb{N}$ is a collection $\mathcal{U}$ of subsets of $\mathbb{N}$ that satisfies the following conditions:
\begin{itemize}
\item the empty set $\emptyset$ does not belong to $\mathcal{U}$;
\item if $A,B\in \mathcal{U}$, then $A\cap B\in \mathcal{U}$;
\item if $A\in \mathcal{U}$ and $A\subset B\subset \mathbb{N}$, then $B\in \mathcal{U}$;
\item if $A\subset \mathbb{N}$, then either $A\in \mathcal{U}$ or $\mathbb{N}\setminus A\in \mathcal{U}$;
\item all the finite subsets of $\mathbb{N}$ do not belong to $\mathcal{U}$.
\end{itemize}  
\end{definition}

\begin{definition}
Let $\mathcal{U}$ be a non-principal ultrafilter on $\mathbb{N}$ and $\{x_n\}_{n\in \mathbb{N}}$ be a bounded sequence of real numbers. The \textit{$\mathcal{U}-limit$} of the sequence, denoted by $\lim_{\mathcal{U}}x_n$, is a real number $x$ satisfying 
      $$\{n\in\mathbb{N}:|x_n-x|\leq \epsilon\}\in \mathcal{U}$$
for any $\epsilon>0$. 
\end{definition}
Note that for any bounded sequence of real numbers, its $\mathcal{U}-limit$ exists and is unique (see Chapter 11 of \cite{SetTheory}). Moreover, let $\{x_n\}_{n\in \mathbb{N}}$ and $\{y_n\}_{n\in \mathbb{N}}$ be two bounded sequences of real numbers, then $\lim_{\mathcal{U}}(x_n+y_n)=\lim_{\mathcal{U}}x_n+\lim_{\mathcal{U}}y_n$ and if $x_n\leq y_n$ for all but finitely many $n\in \mathbb{N}$, then $\lim_{\mathcal{U}}x_n\leq \lim_{\mathcal{U}}y_n$.

\begin{definition} \label{UltraproductHilbert}
Let $(H_n)_{n\in \mathbb{N}}$ be a family of Hilbert spaces and $\mathcal{U}$ be a non-pricipal ultrafilter on $\mathbb{N}$, consider the space
      $$\ell^{\infty}(\mathbb{N},(H_n)_{n\in \mathbb{N}})=\{(x_n)\in \prod_{n\in\mathbb{N}}H_n: \sup_{n\in \mathbb{N}}||x_n||<\infty\}$$
and 
      $$C_0^{\mathcal{U}}(\mathbb{N},(H_n)_{n\in \mathbb{N}})=\{(x_n)\in \ell^{\infty}(\mathbb{N},(H_n)_{n\in \mathbb{N}}): \lim_{\mathcal{U}}||x_n||=0\}.$$
The \textit{ultraproduct} $H_{\mathcal{U}}$ of the family $(H_n)_{n\in \mathbb{N}}$ with respect to the non-principal ultrafilter $\mathcal{U}$ is defined to be the closure of quotient space $\ell^{\infty}(\mathbb{N},(H_n)_{n\in \mathbb{N}})/C_0^{\mathcal{U}}(\mathbb{N},(H_n)_{n\in \mathbb{N}})$ equipped with the inner product 
      $$\langle[x],[y]\rangle_{\mathcal{U}}=\lim_{\mathcal{U}}(\mathrm{Re}\langle x_n,y_n\rangle)+i\lim_{\mathcal{U}}(\mathrm{Im}\langle x_n,y_n\rangle),$$
where $x=(x_n)_{n\in \mathbb{N}},y=(y_n)_{n\in \mathbb{N}}$. 
\end{definition}

We next consider the concept of \textit{r-local affine action} (see \cite{ArntFibredLp}), then we construct an affine action on the ultraproduct of a family of Hilbert spaces from a family of $r$-local affine actions.

\begin{definition}\label{r-local}
Let $r>0$ and $\Gamma$ be a finitely generated group equipped with a word length function $|\cdot|$. A map $\alpha: \Gamma \times H\rightarrow H$ is called an \textit{$r$-local affine action} of $\Gamma$ on $H$, if for any $\gamma \in \Gamma$ and $h\in H$, 
      $$\alpha_{\gamma}(h)=T_{\gamma}(h)+b_{\gamma}$$
where $T: \Gamma \times H \rightarrow H$ and $b: \Gamma \rightarrow H$ are two maps satisfying the following conditions:
\begin{itemize}
\item $T_{\gamma}:H\rightarrow H$ is a linear bijection for all $\gamma\in \Gamma$ with $|\gamma|\leq r$;
\item for all $\gamma_1,\gamma_2\in \Gamma$ with $\max\{|\gamma_1|,|\gamma_2|,|\gamma_1\gamma_2|\}\leq r$, $T_{\gamma_1\gamma_2}=T_{\gamma_1}T_{\gamma_2}$;
\item for all $\gamma_1,\gamma_2\in \Gamma$ with $\max\{|\gamma_1|,|\gamma_2|,|\gamma_1\gamma_2|\}\leq r$, $b_{\gamma_1\gamma_2}=T_{\gamma_1}(b_{\gamma_2})+b_{\gamma_1}.$
\end{itemize}
In addition, if there exists a constant $L>0$ such that $||T_{\gamma}||\leq L$ for all $\gamma\in \Gamma$ with $|\gamma|\leq r$, we call $\alpha$ is an \textit{r-local uniformly bounded affine action}.
\end{definition}

The following lemma plays a crucial role in the proof of our main result.
\begin{lemma}\label{LocalGlobal}
Let $\Gamma$ be a finitely generated group, $(H_n)_{n\in \mathbb{N}}$ be a family of Hilbert spaces and $H_{\mathcal{U}}$ be the ultraproduct of the family $(H_n)_{n\in \mathbb{N}}$ with respect to a non-principal ultrafilter $\mathcal{U}$ on $\mathbb{N}$. If for each $n\in \mathbb{N}$, $\Gamma$ admits an $n$-local uniformly bounded affine action $\alpha^{(n)}$ on $H_n$ with 
      $$\alpha^{(n)}_{\gamma}\cdot=T^{(n)}_{\gamma}\cdot+b^{(n)}_{\gamma}$$
where $\sup_{|\gamma|\leq n,n\in\mathbb{N}}||T^{(n)}_{\gamma}||\leq L$ for some constant $L$ and $\sup_{n\in\mathbb{N}}||b^{(n)}_{\gamma}||<\infty$ for any $\gamma\in \Gamma$, then there exists a uniformly bounded affine action $\alpha$ of $\Gamma$ on $H_{\mathcal{U}}$ such that
      $$\alpha_{\gamma}\cdot=T_{\gamma}\cdot+b_{\gamma}$$
where $T:\Gamma \times H_{\mathcal{U}} \rightarrow H_{\mathcal{U}}$ is a uniformly bounded representation of $\Gamma$ on $H_{\mathcal{U}}$ and $b: \Gamma \rightarrow H_{\mathcal{U}}$ is a map such that
\begin{center}
$T_{\gamma}(h)=(T^{(n)}_{\gamma}(h_n))_{n\in \mathbb{N}}$ and $b_{\gamma}=(b^{(n)}_{\gamma})_{n\in \mathbb{N}}$ 
\end{center}
for $\gamma\in \Gamma$ and $h=[(h_n)_{n\in \mathbb{N}}]\in H_{\mathcal{U}}$.
\end{lemma}
\begin{proof}
We first show that $(T_{\gamma})_{\gamma\in \Gamma}$ is a family of uniformly bounded linear operators on $H_{\mathcal{U}}$. For each $\gamma\in \Gamma$ and any $c,c'\in \mathbb{C}$, $h=[(h_n)_{n\in \mathbb{N}}],h'=[(h'_n)_{n\in \mathbb{N}}]\in H_{\mathcal{U}}$, by Definition \ref{r-local}, we have 
      $$T^{(n)}_{\gamma}(ch_n+c'h'_n)=cT^{(n)}_{\gamma}(h_n)+c'T^{(n)}_{\gamma}(h'_n)$$
for all but finitely many $n\in \mathbb{N}$ which implies that $T_{\gamma}$ is a linear operator. Moreover, for each $\gamma\in \Gamma$, 
      $$||T_{\gamma}(h)||_{\mathcal{U}}=\lim_{\mathcal{U}}||T^{(n)}_{\gamma}(h_n)||\leq L\lim_{\mathcal{U}}||h_n||=L||h||_{\mathcal{U}}$$
which implies that $(T_{\gamma})_{\gamma\in \Gamma}$ is a family of uniformly bounded operators.\par
Similarly, we can show that the action $\alpha$ of $\Gamma$ on $H_{\mathcal{U}}$ is an affine action. 
\end{proof}

\begin{remark}
Actually, we can define ultraproducts for a family of Banach spaces just like the case of Hilbert spaces (Definition \ref{UltraproductHilbert}), then the above lemma holds for  more general Banach spaces which are closed under ultraproducts, such as $L^p$-spaces for $1\leq p <\infty$ and uniformly convex Banach spaces.
\end{remark}

We now state and prove our main theorem for box spaces.
\begin{definition}\label{DefBoxSpace}
A finitely generated group $\Gamma$ is called \textit{residually finite} if there exists a sequence of finite index normal subgroups $\Gamma=\Gamma_0\unrhd \Gamma_1\unrhd \cdots \unrhd \Gamma_n \unrhd \cdots$ such that $\bigcap_{i\in\mathbb{N}}\Gamma_i=\{e\}$, where $e$ is the unit of $\Gamma$. The \textit{box space} associated with the sequence $\{\Gamma_i\}$, denoted by $\square_{\{\Gamma_i\}}\Gamma$ or simply $\square\Gamma$, is the coarse disjoint union $\bigsqcup_{i\in \mathbb{N}}(\Gamma/\Gamma_i)$, where the metric on $\Gamma/\Gamma_i$ is the quotient metric induced from the word length metric on $\Gamma$ for each $i\in \mathbb{N}$.
\end{definition}

\begin{theorem}\label{MainTheorem}
Let $\Gamma$ be a finitely generated residually finite group with a sequence of finite index normal subgroups $\Gamma=\Gamma_0\unrhd \Gamma_1\unrhd \cdots \unrhd \Gamma_n \unrhd \cdots$ such that $\bigcap_{i\in\mathbb{N}}\Gamma_i=\{e\}$, $\square\Gamma$ be the box space associated with this sequence. Then $\Gamma$ is u.b. a-T-menable if and only if $\square\Gamma$ admits a u.b. fibred coarse embedding into Hilbert space.
\end{theorem}
\begin{proof}
We first assume that $\Gamma$ is u.b. a-T-menable. Because $\Gamma$ is residually finite, thus $\{\Gamma\}_{i\in \mathbb{N}}$ is a sequence of asymptotically faithful Galois covering of $\Gamma/\Gamma_i$, then by Theorem \ref{EquivariantAndFibred} and Lemma \ref{ActionImplyEmbed}, the box space $\square\Gamma$ admits a u.b. fibred coarse embedding into Hilbert space.\par
Now we consider the other side of this theorem. For the sake of convenience, we set $X_i=\Gamma/\Gamma_i$. By Definition \ref{ubfibred} and the residually finite property of $\Gamma$, for any integer $r>0$, there exists a positive integer $n_r$ such that (\romannumeral1) $d(X_i,X_j)>r$ for all $i,j\geq n_r$ and for each subset $C\subset X_{n_r}$ of diameter less than $r$ there exists a trivialization $t_C$ satisfying the conditions of Definition \ref{ubfibred}; (\romannumeral2) the quotient map $\pi_{n_r}:\Gamma \rightarrow X_{n_r}$ is an $r$-isometric map. \par
Fixing $n_r$ and $z\in X_{n_r}$, let $C_z=\{x\in X_{n_r}: d(z,x)\leq r\}$. For each $x\in X_{n_r}$, define a vector $c^z_r(x)$ in the Hilbert space $H$ by
$$ c^z_r(x)=\left\{
\begin{array}{lcl}
t_{C_z}(z)(s(z))-t_{C_z}(zx)(s(zx)), & & {d(e,x)\leq r}\\
0,        & & {\text{otherwise}}
\end{array} \right. $$
where $e$ is the identity element of $X_{n_r}$. Let $\rho_{\pm}$ be the control functions coming from Definition \ref{ubfibred}, then we have 
$$\rho_{-}(d(e,x))\leq ||c^z_r(x)||\leq \rho_+(d(e,x))\eqno(*)$$
for any $x\in C_e$. \par
Define the map $b_r: X_{n_r}\rightarrow (\bigoplus_{z\in X_{n_r}}H)$ by
      $$b_r(x)=(c^z_r(x))_{z\in X_{n_r}}$$
for $x\in X_{n_r}$. We equip $(\bigoplus_{z\in X_{n_r}}H)$ with the inner product $\langle(h_z),(h'_z)\rangle=|X_{n_r}|^{-1}\sum_{z\in X_{n_r}}\langle h_z,h'_z\rangle$, where $|X_{n_r}|$ is the cardinal of the finite set $X_{n_r}$, then by $(*)$, we have 
      $$\rho_-(d(e,x))\leq||b_r(x)||\leq \rho_+(d(e,x))\eqno(**)$$ 
for any $x\in C_e$.\par
For $x\in X_{n_r}$, define $T_r(x):(\bigoplus_{z\in X_{n_r}}H)\rightarrow (\bigoplus_{z\in X_{n_r}}H)$ by
$$T_r(x)(h)=\left\{
\begin{array}{lcl}
(t_{C_zC_{zx}}(h_{zx}))_{z\in X_{n_r}}, & & {x\in C_e}\\
h,    & & {\text{otherwise}}
\end{array} \right.$$
where $h=(h_z)_{z\in X_{n_r}}$. Then $T_r(x)$ is a uniformly bounded linear operator on $(\bigoplus_{z\in X_{n_r}}H)$, i.e. for any $x\in X_{n_r}$, $L^{-1}\leq ||T_r(x)||\leq L$, where $L$ is the constant coming from Definition \ref{ubfibred}. And we have $T_r(xy)=T_r(x)T_r(y)$ for all $x,y,xy\in C_e$. \par
Moreover, for $x,y,xy\in C_e$, by the relation $t_{C_1C_2}=t_{C_1}(x')t_{C_2}(x')^{-1}$ for all $x'\in C_1\cap C_2$, we have 
      $$b_r(xy)=T_r(x)(b_r(y))+b_r(x).\eqno(***)$$  \par
Then we define $\Tilde{T}_r=T_r\circ\pi_{n_r}$ and $\Tilde{b}_r=b_r\circ\pi_{n_r}$ to be the lift of $T_r$ and $b_r$ to the $r$-ball $\{\gamma\in\Gamma: d(e,\gamma)\leq r\}$ of $\Gamma$, respectively, and define $\Tilde{T}_r=\mathrm{Id}$, $\Tilde{b}_r=0$ outside the ball. Furthermore, for $\gamma \in \Gamma$, define the map $\alpha_r(\gamma):(\bigoplus_{z\in X_{n_r}}H)\rightarrow (\bigoplus_{z\in X_{n_r}}H)$ by 
      $$\alpha_r(\gamma)(h)=\Tilde{T}_r(\gamma)(h)+\Tilde{b}_r(\gamma)$$
for $h\in (\bigoplus_{z\in X_{n_r}}H)$. Then by $(***)$ and the fact of the quotient map $\pi_{n_r}$ is an $r$-isometry, we obtain that $\alpha_r$ is an $r$-local uniformly bounded affine action on the Hilbert space $(\bigoplus_{z\in X_{n_r}}H)$. \par
Therefore, by the right hand side of $(**)$ and Lemma \ref{LocalGlobal}, there exists a uniformly bounded affine action $\alpha$ of $\Gamma$ on the ultraproduct $(\bigoplus_{z\in X_{n_r}}H)_{\mathcal{U}}$ with respect to a non-principal ultrafilter $\mathcal{U}$ on $\mathbb{N}$. Furthermore, by the left hand side of $(**)$ and Lemma \ref{LocalGlobal}, we obtain that the action $\alpha$ is metrically proper. In conclusion, $\Gamma$ is u.b. a-T-menable.   
\end{proof}
\begin{remark}\label{aTmenableAndFibred}
In particular, if we replace ``uniformly bounded" by ``isometric" in the above proof, then we obtain that a finitely generated residually finite group is a-T-menable if and only if one (or equivalently, all) of its box space admits a fibred coarse embedding into Hilbert space. This result has been proven by X. Chen, Q. Wang and X. Wang in \cite{ChenWangWangFibredCoarse}.
\end{remark}

We now consider the above theorem for warped cones introduced by J. Roe in \cite{RoeWarpedCone}\cite{RoeFoliation}.

\begin{definition}
Let $(X,d)$ be a metric space and $\Gamma$ be a finitely generated group with a symmetric generating set $S$, acting on $X$ by homeomorphisms. The \textit{warped metric} $d_{\Gamma}$ is the greatest metric satisfying the following inequalities
$$d_{\Gamma}(x,x')\leq d(x,x'), \:\:\:\:\:\: d_{\Gamma}(x,sx)\leq 1  $$
for any $x,x'\in X$, $s\in S$.
\end{definition}

The warped metric can be characterized by the following lemma (see Proposition 1.6 of \cite{RoeWarpedCone}).

\begin{lemma}
Let $|\cdot|$ be the word length function on $\Gamma$ relative to the generating set $S$, then for any $x,y\in X$, the warped distance from $x$ to $y$ is the infimum of all sums
    $$\sum (d(\gamma_ix_i,x_{i+1})+|\gamma_i|)$$  
taken over all finite sequences $x=x_0,x_1,\cdots,x_N=y$ in $X$ and $\gamma_0,\gamma_1,\cdots,\gamma_{N-1}$ in $\Gamma$. Moreover, assume $X$ to be a proper metric space, if $d(x,y)\leq k$, then one can find a finite sequence as above with $N=k+1$ such that the infimum is attained.
\end{lemma}

Note that warped metrics relative to two different generating sets of $\Gamma$ are not the same, however, they are coarsely equivalent by the above lemma. 

\begin{definition} \label{DefWarpedCones}
Let $(Y,d_Y)$ be a compact space of diameter at most 2, admitting a continuous action of a finitely generated group $\Gamma$. The \textit{cone} of $Y$ is the metric space $\mathcal{O}Y=([1,\infty)\times Y,d)$, where metric $d$ is defined as follows:
      $$d((t,y),(t',y')):=|t-t'|+\min\{t,t'\}\cdot d_Y(y,y')$$
The \textit{warped cone} of $Y$, denoted by $\mathcal{O}_{\Gamma}Y$, is the cone of $Y$ with the warped metric, where the warping group action is defined by the formula: $\gamma(t,y):=(t,\gamma y)$.
\end{definition}

\begin{remark}
Let $Y$ be a compact smooth manifold, embed $Y$ smoothly into a high-dimensional sphere $S^{N-1}$, and let $X$ be the union of all the rays through the origin in $\mathbb{R}^N$ that meet the embedded copy of $Y$. Equip $X$ with the metric induced from $\mathbb{R}^N$, then $X$ is coarsely equivalent to $\mathcal{O}Y$.
\end{remark}


We now introduce the concept of the straightening warped cone defined by D. Sawicki and J. Wu in \cite{SawickiWu}, which is the Galois covering of the warped cone.
\begin{definition}(\cite{SawickiWu})
Let $(X,d)$ be a metric space with a continuous action of a finitely generated group $\Gamma$, with $S$ being a finite set of generators. The \textit{twisted metric} $d^1$ on the Cartesian product $\Gamma\times X$ is the largest metric such that 
    $$d^1((\gamma,x),(s\gamma,x))=1, \:\:\:\:\:\: d^1((\gamma,x),(\gamma,x'))\leq d(\gamma x,\gamma x')$$
for all $s\in S\backslash\{e\}$, $\gamma\in \Gamma$ and $x,x'\in X$.\par
Let $(Y,d_Y)$ be a compact space of diameter at most 2, admitting a continuous action of a finitely generated group $\Gamma$. The \textit{straightening warped cone} of $Y$ is the product space $\Gamma \times \mathcal{O}Y$ with the twisted metric $d^1$, where the group action of $\Gamma$ on $\mathcal{O}Y$ is defined by $\gamma(t,y):=(t,\gamma y)$.
\end{definition} 

Note that the action of $\Gamma$ on $(\Gamma\times X,d^{1})$ is given by $\gamma(\eta,x):=(\eta\gamma^{-1},\gamma x)$, which is a free, proper, isometric action and that the quotient space of this action can be isometrically identified as $X$ via the quotient map $(\gamma,x)\mapsto \gamma x$ (see Proposition 3.7 of \cite{SawickiWu})

The following definition is a more general case of Definition \ref{asym-faithful}.
\begin{definition}
A surjective map $q:Z\rightarrow X$ between metric spaces is said to be \textit{asymptotically faithful} if for every $N\in \mathbb{N}$ there is a subset $A\subset X$ with a bounded complement such that for any $z\in q^{-1}(A)$, the map $q$ restricts to an isometry between the balls $B(z,N)$ and $B(q(z),N)$.
\end{definition}

The following lemma is the main result in \cite{SawickiWu}.
\begin{lemma}\label{AsymFaithWarpedCone}
The quotient map $(\Gamma \times \mathcal{O}Y,d^1)\rightarrow \mathcal{O}_{\Gamma}Y$ given by $(\gamma,t,y)\mapsto (t,\gamma y)$ is asymptotically faithful if and only if the action $\Gamma \curvearrowright Y$ is free.
\end{lemma}
\begin{proof}
Refer to the proof of Proposition 3.10 in \cite{SawickiWu}.
\end{proof}

Before stating our main theorem for warped cones, we introduce two notions.

\begin{definition}
An action $\Gamma \curvearrowright Y$ is \textit{linearizable} in a Banach space $E$ if there exists an isometric representation of $\Gamma$ on $E$ and a bi-Lipschitz equivariant embedding $Y\rightarrow E$.
\end{definition}

\begin{definition}
The action of the group $G$ on the measure space $(X,\mu)$ is called \textit{essentially free} if $\mu(X_0)=0$, where $X_0$ be the subset of $X$ such that $x\in X_0$ if and only if $gy=y$ for some non-identity element $g \in G$.
\end{definition}

Now we state our main theorem for warped cones, and its proof is similar to the proof of Theorem \ref{MainTheorem}.
\begin{theorem} \label{MainTheorem2}
Let $\Gamma \curvearrowright Y$ be an action of a finitely generated group on a compact metric space.
\begin{enumerate}
\item Assume the action is free and linearizable in a Hilbert space if $\Gamma$ is u.b. a-T-menable, then $\mathcal{O}_{\Gamma}Y$ admits a u.b. fibred coarse embedding into Hilbert space.
\item Assume the action admits an invariant Borel probability measure $\mu$ on $Y$ and is essentially free with respect to $\mu$, if $\mathcal{O}_{\Gamma}Y$ admits a u.b. fibred coarse embedding into Hilbert space, then $\Gamma$ is u.b. a-T-menable.
\end{enumerate}
\end{theorem}

\begin{proof}
For  statement $(1)$. By Lemma \ref{ActionImplyEmbed}, $\Gamma$ admits an $L$-controlled equivariant coarse embedding into some Hilbert space $H$ for some $L>0$. Moreover, there exists a bi-Lipschitz equivariant embedding $Y\rightarrow H_0$ since the action is linearizable. Combining them, we obtain an $L$-controlled equivariant coarse embedding from the straightening warped cone $(\Gamma \times \mathcal{O}Y,d^1)$ to the Hilbert space $H\oplus C \oplus H_0$. Then statement $(1)$ follows from  Theorem \ref{EquivariantAndFibred} and Lemma \ref{AsymFaithWarpedCone} (note that we just state and prove Theorem \ref{EquivariantAndFibred} for the coarse disjoint union, but it is true for the general metric space by a similar proof). \par

For statement $(2)$. Let $Y_0$ be the subset of $Y$ such that $y\in Y_0$ if and only if $\gamma y=y$ for some non-identity element $\gamma \in \Gamma$, then $\mu(Y_0)=0$ and the group action of $\Gamma$ on $Y\backslash Y_0$ is free. Thus we can assume the action on $Y$ is free. Otherwise, we replace $Y$ by $Y\backslash Y_0$. By Definition \ref{ubfibred} and Lemma \ref{AsymFaithWarpedCone}, for any integer $r>0$, there exists a positive integer $t_r>1$ such that (\romannumeral1) for each subset $C\subset t_r\times Y\subset \mathcal{O}_{\Gamma}Y$ of diameter less than $r$, there exists a trivialization $t_C$ satisfying the conditions of Definition \ref{ubfibred}; (\romannumeral2) the quotient map $q_{t_r}:(\Gamma\times t_r\times Y,d^1)\rightarrow (t_r\times Y,d_{\Gamma})\subset \mathcal{O}_{\Gamma}Y$ given by $(\gamma,t_r,y)\mapsto (t_r,\gamma y)$ is an $r$-isometric map (here we need freeness of the action from Lemma \ref{AsymFaithWarpedCone}). \par

Now fixing $t_r$, we identify the metric space $(t_r\times Y,d_{\Gamma})$ with the space $Y$ equipped with the warped metric $d_{\Gamma}$ of $t_r\times Y$. For $y\in Y$, let $C_y:=\{y'\in Y: d_{\Gamma}(y,y')\leq r\}$, for each $\gamma\in \Gamma$, define a vector $c^{y}_r(\gamma)$ in the Hilbert space $H$ by
      $$ c^y_r(\gamma)=\left\{
\begin{array}{lcl}
t_{C_y}(y)(s(y))-t_{C_y}(\gamma^{-1} y)(s(\gamma^{-1} y)), & & {d(e,\gamma)\leq r}\\
0,        & & {\text{otherwise}}
\end{array} \right. $$ 
where $e$ is the identity element of $\Gamma$. When $d(e,\gamma)\leq r$, we have $d_{\Gamma}(y,\gamma^{-1} y)=d^1((e,y),(\gamma^{-1},y))=d(e,\gamma^{-1})\leq r$ by (\romannumeral2) above, this means $\gamma^{-1} y\in C_y$, thus the map $c^{y}_r: \Gamma \rightarrow H$ is well-defined. Let $\rho_{\pm}$ be the control functions coming from Definition \ref{ubfibred}, then we have 
      $$\rho_{-}(d(e,\gamma))\leq ||c^{y}_r(\gamma)||\leq \rho_+(d(e,\gamma))  \eqno(*)$$
for any $\gamma$ with word length less than $r$. \par
Define the map $b_r: \Gamma\rightarrow L^2(Y,\mu;H)$ by
      $$b_r(\gamma): y\mapsto c^{y}_r(\gamma)$$
for $\gamma \in \Gamma$. We equip $L^2(Y,\mu;H)$ with the inner product $\langle h,h'\rangle=\int_{Y}\langle h(y),h'(y)\rangle\, d(\mu)$, then by $(*)$, we have 
      $$\rho_-(d(e,\gamma))\leq||b_r(\gamma)||\leq \rho_+(d(e,\gamma))\eqno(**)$$ 
for any $\gamma$ with word length less than $r$. \par
For $\gamma \in \Gamma$, define $T_r(\gamma): L^2(Y,\mu;H) \rightarrow L^2(Y,\mu;H)$ by
$$T_r(\gamma)(h): y\mapsto \left\{
\begin{array}{lcl}
t_{C_yC_{\gamma^{-1} y}}(h(\gamma^{-1} y)), & & {d(e,\gamma)\leq r}\\
h(y),    & & \text{otherwise}
\end{array} \right.$$
where $h\in L^2(Y,\mu;H)$. Then for any $\gamma \in \Gamma$, $L^{-1}\leq ||T_r(\gamma)||\leq L$, where $L$ is the constant coming from Definition \ref{ubfibred}. And we have $T_r(\gamma_1\gamma_2)=T_r(\gamma_1)T_r(\gamma_2)$ for all $\gamma_1$,$\gamma_2$,$\gamma_1\gamma_2$ with word length less than $r$. \par
For $\gamma\in \Gamma$, define 
      $$\alpha_r(\gamma)(h):=T_r(\gamma)(h)+b_r(\gamma)$$
where $h\in L^2(Y,\mu;H)$, then $\alpha_r$ is an $r$-local uniformly bounded affine action on the Hilbert space $L^2(Y,\mu;H)$. \par
Therefore, by the right side of $(**)$ and Lemma \ref{LocalGlobal}, there exists a uniformly bounded affine action $\alpha$ of $\Gamma$ on the ultraproduct $(L^2(Y,\mu;H))_{\mathcal{U}}$ with respect to a non-principal ultrafilter $\mathcal{U}$ on $\mathbb{N}$. Furthermore, by the left side of $(**)$ and Lemma \ref{LocalGlobal}, we obtain that the action $\alpha$ is metrically proper. In conclusion, $\Gamma$ is u.b. a-T-menable.  
\end{proof}

\begin{remark}
The above theorem for the ``isometric'' case (i.e. $L=1$) has been proven by D. Sawicki, J. Wu in \cite{SawickiWu} and Q. Wang, Z. Wang in \cite{2WangWarpedCone}. 
\end{remark}

\begin{remark}
In fact, our two main theorems (i.e. Theorem \ref{MainTheorem}, \ref{MainTheorem2}) hold for the more general Banach space which is closed under ultra-products, such as $L^p$-spaces for $1\leq p <\infty$ and uniformly convex Banach spaces. 
\end{remark}

\section{Examples}\label{Examples}
\begin{example}
The following are some known results.
\begin{enumerate}
\item Obviously, a metric space that can be coarsely embedded into Hilbert space admits a (u.b.) fibred coarse embedding into Hilbert space, such as hyperbolic groups and linear groups as metric spaces (see Chapter 5 of \cite{NowakYu}).
\item The free group $\mathbb{F}_2$ on two generators and the special linear group $SL(2,\mathbb{Z})$ are a-T-menable (see \cite{HaagerupProperty}), then by Theorem \ref{MainTheorem}, all box spaces $\square \mathbb{F}_2$ and $\square SL(2,\mathbb{Z})$ admit a (u.b.) fibred coarse embedding into Hilbert space. But there exist such box spaces which cannot be coarsely embedded into Hilbert space since they could be expander graphs (see Chapter 5 of \cite{NowakYu}, Chapter 13 of \cite{WillettYuBook}). 
\item A space of graphs with large girth admits a (u.b.) fibred coarse embedding into Hilbert space, such as a sequence of Ramanujan graphs (see \cite{ChenWangYufibredcoarse}\cite{WillettYuExpander}).
\end{enumerate}
\end{example}

The following examples illustrate the differences between u.b. fibred coarse embeddability and fibred coarse embeddability.
\begin{example} \label{Example2}
Let $\mathbb{H}$ be the field of quaternions. Define a sesquilinear form $q$ on $\mathbb{H}^{n+1}$ by
      $$q(v,w)=-\overline{v}_0w_0+\sum_{i=1}^n \overline{v}_iw_i$$
where $v=(v_0,\cdots,v_n)^T, w=(w_0,\cdots,w_n)^T$ are two elements in $\mathbb{H}^{n+1}$. For $n\geq 1$, define $Sp(n,1)$ be a collection of $(n+1)\times(n+1)$ matrices over $\mathbb{H}$ which acts on $\mathbb{H}^{n+1}$ from left and preserves the sesquilinear form $q$, i.e. for each $A\in Sp(n,1)$, we have  
      $$q(Av,Aw)=q(v,w)$$
for any $v,w\in \mathbb{H}^{n+1}$. \par
Every cocompact lattice $\Gamma$ of $Sp(n,1)$ is finitely generated residually finite group by Mal\'{e}v's theorem, thus Theorem \ref{MainTheorem} and  Theorem \ref{NishikawaTheorem} tell us that each of its box spaces $\square \Gamma$ admits a u.b. fibred coarse embedding into Hilbert space. On the other hand, $Sp(n,1)$ has property(T) for any $n\geq 2$, so does $\Gamma$ (see \cite{PropertyT}). It implies that $\Gamma$ is not a-T-menable. Thus all box spaces $\square \Gamma$ do not admit a fibred coarse embedding into Hilbert space. \par
Furthermore, in \cite{ChenWangYufibredcoarse}, X. Chen, Q. Wang and G. Yu proved that the maximal coarse Baum-Connes conjecture holds for metric spaces  admitting a fibred coarse embedding into Hilbert space. However, for space admitting a u.b. fibred coarse embedding into Hilbert space, this conjecture maybe not true. For example, consider the above box space $\square \Gamma$. Because $\Gamma$ has property(T) for $n\geq 2$, thus $\square \Gamma$ has geometric property(T), then by a theorem proved by R. Willett and G. Yu (see Theorem 7.4 in \cite{WillettYuExpander2}), the maximal coarse Baum-Connes conjecture does not hold for $\square \Gamma$.  
\end{example}

\begin{example}
Let $\Gamma$ be a cocompact lattice in $Sp(n,1)$ as above, then there exists a sequence of finite index normal subgroups $\Gamma=\Gamma_0\unrhd \Gamma_1\unrhd \cdots \unrhd \Gamma_m \unrhd \cdots$ such that $\bigcap_{i\in\mathbb{N}}\Gamma_i=\{e\}$. Let $q: \Gamma/\Gamma_{m+1} \rightarrow \Gamma/\Gamma_m$ be the quotient map, then we have an inverse system $\{\Gamma/\Gamma_m,q\}$ and its inverse limit is 
      $$\underleftarrow{\lim}(\Gamma/\Gamma_m) = \{(g_m)\in \prod(\Gamma/\Gamma_m): q(g_{m+1})=g_m\}.$$
Define a metric $d$ on $\underleftarrow{\lim}(\Gamma/\Gamma_m)$ by
      $$d((g_m),(g'_m)):=1/2^j$$
where $j$ is the smallest index such that $g_j\neq g'_j$. Consequently, $(\underleftarrow{\lim}(\Gamma/\Gamma_m),d)$ is a compact metric space and the action $\Gamma \curvearrowright \underleftarrow{\lim}(\Gamma/\Gamma_m)$ given by $\gamma\cdot (g_m)=(\gamma g_m)$ is free. Now we consider the Hilbert space $\ell^2(\bigsqcup(\Gamma/\Gamma_m))$ with the isometric $\Gamma$-action given by $\gamma\cdot\delta_{g}=\delta_{\gamma g}$, where $g\in \bigsqcup(\Gamma/\Gamma_m)$. Define a map $f: \underleftarrow{\lim}(\Gamma/\Gamma_m)\rightarrow \ell^2(\bigsqcup(\Gamma/\Gamma_m))$ by $f((g_m))=\sum_m \delta_{g_m}/2^m$, then $f$ is a $\Gamma$-equivariant map and we have $||f((g_m))-f((g'_m))||=(2\sqrt{6}/3)\cdot d((g_m),(g'_m))$. Thus the action $\Gamma \curvearrowright \underleftarrow{\lim}(\Gamma/\Gamma_m)$ is linearizable in the Hilbert space $\ell^2(\bigsqcup(\Gamma/\Gamma_m))$.\par
Finally, by Theorem \ref{MainTheorem2} and Theorem \ref{NishikawaTheorem}, the warped cone $\mathcal{O}_{\Gamma}(\underleftarrow{\lim}(\Gamma/\Gamma_m))$ admits a u.b. fibred coarse embedding into Hilbert space, but it does not admit a fibred coarse embedding into Hilbert space since $\Gamma$ has property(T) when $n\geq 2$.
\end{example}

\bibliographystyle{plain}
\bibliography{Ubfce}

\end{document}